\documentclass[reqno]{amsart}

\usepackage{amsmath, amssymb, latexsym, amscd, amsthm,amsfonts,amstext}
\usepackage[mathscr]{eucal}
\usepackage{graphics}

\theoremstyle{plain}

\newtheorem{lemma}{Lemma}[section]
\newtheorem{theorem}[lemma]{Theorem}
\newtheorem{corollary}[lemma]{Corollary}
\newtheorem{conjecture}[lemma]{Conjecture}

\newtheorem{proposition}[lemma]{Proposition}
\newtheorem{definition}[lemma]{Definition}
\newtheorem{remark}{Remark}[section]
\newtheorem{example}[lemma]{Example}
\newtheorem*{Acknowlegement}{Acknowlegement}
\newenvironment{proof*}{\vskip 2mm\noindent {}}{\hfill $\Box$ \vskip 2mm}

\begin{document}

\title[On limit Brody curves in $\mathbb C^n$ and $(\mathbb C^*)^2$]{On limit Brody curves in $\mathbb C^n$ and $(\mathbb C^*)^2$}

\thanks{The research of the authors is supported by an NAFOSTED grant of Vietnam.}

\author[]{*Do Duc Thai, *Mai Anh Duc \and **Ninh Van Thu}

\thanks{The research of the third author was supported in part by an NRF grant
2011-0030044 (SRC-GAIA) of the Ministry of Education, The Republic of Korea.}

\address{*Department of Mathematics \newline
Hanoi National University of Education\newline
136 Xuan Thuy str., Hanoi, Vietnam}

\email{ ducthai.do@gmail.com, ducphuongma@gmail.com}
\address{**Department of Mathematics and Center for Geometry and its Applications,
Pohang University of Science and Technology,  Pohang 790-784, Republic of Korea}
\email{thunv@postech.ac.kr, thunv@vnu.edu.vn}


\begin{abstract} 
In this paper, the conjecture on the Zalcmanness of $\mathbb C^n \ (n\geq 2)$ and $(\mathbb C^*)^2$, which is posed in  \cite{Do}, is proved in the case where  the derivatives of limit holomorphic curves are bounded. Moreover, several criteria for normality of families of holomorphic mappings are given.
\end{abstract}

\maketitle

\section{Introduction}
The main aim of this paper is to study the conjecture on the Zalcmanness of $\mathbb C^n \ (n\geq 2)$, which is posed in  \cite{Do}. 
We now recall the above conjecture.

\begin{definition}\label{Def20} (see \cite[Def. 2.9]{Do}) Let $X$ be a complex space. The complex space $X$ is said to be a Zalcman complex spsace if $X$ satisfies the following:

For each non-normal family $\mathcal{F}\subset \mathrm{Hol}(\Delta, X)$ such that $\mathcal{F}$ is not compactly divergent, then there exist sequences $\{p_j\}\subset \Delta$ with $p_j\to p_0\in\Delta$ as $j\to\infty$, $\{f_j\} \subset \mathcal{F}, \{\rho_j\}\subset \mathbb R$ with $\rho_j>0$ and $\rho_j \to 0^+$ as $j\to \infty$ such that 
$$
g_j(\xi):=f_j(p_j+\rho_j\xi), \xi\in \mathbb  C,
$$
converges uniformly on any compact subsets of $\mathbb C$ to a non-constant holomorphic curve $g: \mathbb C\to X$.
\end{definition}

Remark that there are numerous examples of Zalcman spaces such as compact complex spaces and the complement of any hyperbolic hypersurface in a compact complex space (see \cite{Do} and \cite{TT}).

\begin{conjecture}\label{Con} (see \cite[Remark 2.14]{Do}) $\mathbb C^n$  is a Zalcman space for each $n \geq 2.$
\end{conjecture}

As far as we know, Conjecture \ref{Con} is still open. In this paper, we show that this conjecture holds if the derivatives of the holomorphic curves $g: \mathbb C\to X$ in Definition \ref{Def20} are bounded.  To state our main result, we give some definitions.

\begin{definition}\label{Def3}(see \cite[pp. 8-10]{Lang}) A length function on a complex space $X$ is a real-valued non-negative continuous function $E$ defined on the tangent bundle $TX$ satisfying
$$
E(v)=0~\text{iff}~v=0, 
$$
and
$$
E(av)=|a|\cdot E(v)~\text{for}~a\in \mathbb C~\text{and}~v\in TX.
$$
\end{definition}

\begin{definition}\label{Def1}
Let $X$ be a hermitian complex space with a length function $E$. A holomorphic curve $f: \mathbb C\to X$ is said to be an $E$-Brody curve if its derivative is bounded, i.e.,  $|f'(z)|_E\lesssim 1$ on $\mathbb C$. Specifically, if $X$ is a domain in $\mathbb P^n(\mathbb C)$, we understand that a Brody curve in $X$ is a $ds^2_{FS}$-Brody curve, where $ds^2_{FS}$ is the Fubini-Study metric on $\mathbb P^n(\mathbb C).$
\end{definition}

For results concerning Brody curves, we refer the reader to the monographs \cite{B-E}, \cite{C-D}, \cite{Duval}, \cite{E1}, \cite{Tsu}, \cite{Win},
\cite{Win1}.  

\begin{definition}\label{Def2} Let $X$ be a complex space with a length function $E$. The complex space $X$ is said to be of $E$-limit type 
 if $X$ satisfies the following:

For each non-normal family $\mathcal{F}\subset \mathrm{Hol}(\Delta, X)$ such that $\mathcal{F}$ is not compactly divergent, then there exist sequences $\{p_j\}\subset \Delta$ with $p_j\to p_0\in\Delta$ as $j\to\infty$, $\{f_j\} \subset \mathcal{F}, \{\rho_j\}\subset \mathbb R$ with $\rho_j>0$ and $\rho_j \to 0^+$ as $j\to \infty$ such that 
$$
g_j(\xi):=f_j(p_j+\rho_j\xi), \xi\in \mathbb  C,
$$
converges uniformly on any compact subsets of $\mathbb C$ to a non-constant $E$-Brody curve $g: \mathbb C\to X$.
\end{definition}

We now give the following.

\begin{theorem}\label{Prop 4.3}
$\mathbb C^n~(n\geq 2)$ is not of $E$-limit type for any  length function $E$ on $\mathbb C^n$. 
\end{theorem}
\begin{theorem}\label{T2} $(\mathbb C^*)^2$ is not of $ds^2_{FS}$-limit type, where $ds^2_{FS}$ is the Fubini-Study metric on $\mathbb P^2(\mathbb C)$.
\end{theorem}

The concept of normal family was first introduced in 1907 by P. Montel \cite{Montel} and generalized by O. Lehto and K. I. Virtanen \cite{L-V}. Since that time the subject of normal maps has been studied intensively (see \cite{A-K, Ber, Do, Wu, Zal} and references therein). In particular, in \cite{Do}, the authors showed criteria for normality of a family of holomorphic mappings in several complex variables into a complete hermitian complex space in the term of the non-constant limit curves.

The Marty's criterion (see \cite[Theorem 17, p. 226]{Al}) asserts that the normality of a family $\mathcal{F}$ of meromorphic functions on a plane domain $D\subset \mathbb C$ is equivalent to the local boundedness of the corresponding family $\mathcal{F}^\# $ of spherical derivatives $f^\# = | f'|/(1 + | f |^2)$. 

The next aim of this article is to generalize the Marty's criterion to several complex variables. Namely, we show the following theorem on the normality of families of holomorphic mappings in the term of the non-constant $E$-Brody curves.

\begin{theorem}\label{T1}
Let $\Omega$ be a domain in $\mathbb C$ and let $M$ be a complete hermitian complex space with a hermitian metric $E$. Let $\mathcal{F}\subset \mathrm{Hol}(\Omega, M)$. Then the family $\mathcal{F}$ is not normal if and only if there exist sequences $\{p_j\}\subset \Omega$ with $p_j\to p_0\in\Omega$ as $j\to\infty$, $\{f_j\} \subset \mathcal{F}, \{\rho_j\}\subset \mathbb R$ with $\rho_j>0$ and $\rho_j \to 0^+$ as $j\to \infty$ such that 
$$
g_j(\xi):=f_j(p_j+\rho_j\xi), \; \xi\in \mathbb  C
$$
satisfies one of the following two assertions
\begin{itemize}
\item[(i)] The sequence $\{g_j\}$ is compactly divergent on $\mathbb C$;
\item[(ii)] The sequence $\{g_j\}$ converges uniformly on compact subsets of $\mathbb C$  to a non-constant $E$-Brody curve 
$g: \mathbb C\to M$. In this case, the curve $g$ is called to be a limit Brody curve with respect to the hermitian metric $E,$ or shortly,  a limit 
$E$-Brody curve.
 \end{itemize}
\end{theorem}

\begin{theorem}\label{T3}
Let $\Omega$ be a domain in $\mathbb C$. Let $X$ be a compact complex space with a hermitian metric $E$. Let $S$ be a complex hypersurface in $X$ and let $M=X\setminus S$. Let  $\mathcal{F}\subset \mathrm{Hol}(\Omega, M)$. Then the family $\mathcal{F}$ is not normal if and only if there exist sequences $\{p_j\}\subset \Omega$ with $p_j\to p_0\in\Omega$ as $j\to\infty$, $\{f_j\} \subset \mathcal{F}, \{\rho_j\}\subset \mathbb R$ with $\rho_j>0$ and $\rho_j \to 0^+$ as $j\to \infty$ such that 
$$
g_j(\xi):=f_j(p_j+\rho_j\xi), \; \xi\in \mathbb  C
$$
satisfies one of the following two assertions
\begin{itemize}
\item[(i)] The sequence $\{g_j\}$ is compactly divergent on $\mathbb C$;
\item[(ii)] The sequence $\{g_j\}$ converges uniformly on compact subsets of $\mathbb C$  to a non-constant $E$-Brody curve $g: \mathbb C\to M$. 
 \end{itemize}
\end{theorem}

\begin{remark} (i)\ Theorem \ref{T1} is also a generalization of  Brody's theorem \cite{R.B} and Zalcman's theorem \cite{Ber, Zal}. 

\noindent
(ii)\ In Theorem \ref{T1}, the hermitian metric $E$ must be complete. Fortunately, Theorem \ref{T3} shows that the assertion still holds in the case where the hermitian metric maybe is not complete.
\end{remark}

This paper is organized as follows. In Section $2$, we will prove Theorems \ref{Prop 4.3} and \ref{T2}. Then the proofs of Theorems 
\ref{T1} and  \ref{T3}  will be given in Section $3$.

\begin{Acknowlegement} The authors would like to thank Prof. Kang-Tae Kim for his precious discussions on this material. 
\end{Acknowlegement}

\section{Non-existence of limit Brody curves in $\mathbb C^n$ and $(\mathbb C^*)^2$}
First of all, we give the following.

\vskip0.2cm
\noindent
{\it Proof of Theorem \ref{Prop 4.3}}

Let $g: \mathbb C \to \mathbb C^{n-1}$ be a holomorphic function such that the function $E\Big(f(z), f'(z)\Big)$ is not bounded on $\mathbb C$, where $f:\mathbb C\to \mathbb C^n\hookrightarrow\mathbb P^n(\mathbb C)$ is the holomorphic function defined by $f(z)=(z,g(z))=[1:z:g_1(z):\cdots:g_{n-1}(z)]\in \mathbb P^n(\mathbb C) $ for all $z \in\mathbb C$. Now denote by $\{f_m\}\subset \mathrm{Hol}(\Delta, \mathbb C^n)$ the sequence of holomorphic discs given by
$$
f_m(z):=f(mz)=(mz,g(mz))=[1:mz:g_1(mz):\cdots:g_{n-1}(mz)]
$$ 
for every $z\in \Delta$.
Suppose that there exist a sequence $\{n_k\}\subset \mathbb N$, a sequence $\{p_k\}\Subset \Delta$, and a sequence $\{\rho_k\}\subset (0,+\infty) $ with $\rho_k \to 0^+ $ as $k\to \infty$ such that the sequence $\{\varphi_k\}$ defined by
\begin{equation*}
\begin{split}
\varphi_k(\xi):&=f_{m_k}(p_k+\rho_k \xi)=f(m_k p_k+m_k\rho_k \xi)\\
              &=(m_k p_k+m_k\rho_k \xi,g(m_k p_k+m_k\rho_k \xi))
\end{split}
\end{equation*}
for each $k\in \mathbb N^*$ and for $|\xi|< 1/\rho_k $, converges uniformly on every compact subset of $\mathbb C$ to a non-constant holomorphic curve $\varphi: \mathbb C\to \mathbb C^n$. Then $m_kp_k\to q\in \mathbb C$ and $m_k\rho_k\to A\in \mathbb C^*$ as $k\to \infty$. Thus, we obtain that $\varphi(\xi)=(q+A\xi,g(q+A\xi))=[1:q+A\xi:g_1(q+A\xi):\cdots :g_{n-1}(q+A\xi)]$ for all $\xi \in \mathbb C$. We note that $E(\varphi(z),\varphi'(z))$ is not bounded on $\mathbb C$. This completes the proof. \hfill $\boxed{ }$

\begin{corollary}\label{Cor2}
$\mathbb C^n~(n\geq 2)$ is not of $ds^2_{FS}$-limit type.
\end{corollary}

\begin{remark} 

\noindent
i) By Theorem \ref{Prop 4.3}, there does not exist a sequence $\{\varphi_k\}$ which converges  uniformly on every compact subset of $\mathbb C$ to any Brody curve in $\mathbb C^n$. But, there is a sequence $\{\varphi_k\}$ which converges  uniformly on every compact subset of $\mathbb C$ to a Brody curve in $\mathbb P^n(\mathbb C)$ (cf. Theorem \ref{T1}).

\noindent
ii) We can see that $T_r(f)\approx T_r(g)$, where $T_r(f)$ is the
Nevanlinna-Cartan characteristic function of $f$ (see \cite[Theorem 2.5.12, p. 64]{N-W}). Moreover since $\varphi(\xi)=(q+A\xi,g(q+A\xi))=[1:q+A\xi:g_1(q+A\xi):\cdots :g_{n-1}(q+A\xi)]$ for all $\xi \in \mathbb C$, it follows that $T_r(\varphi)\approx T_r(f)\approx T_r(g)$.

\noindent
iii) In \cite{Do}, the authors proved that the complement of any hyperbolic hypersurface in a compact complex space are Zalcman. In particular, $\mathbb C^*$ and $\mathbb C$ are Zalcman. Corollary \ref{Cor2} showed that $\mathbb C^n~(n\geq 2)$ is not of $ds^2_{FS}$-limit type. However, Conjecture \ref{Con} remains still open. 

\end{remark}

Attempting to prove Theorem \ref{T2}, we recall the Winkelmann's construction of compact  complex torus $T$, domains $\Omega_1, \Omega_2$ with $\Omega_2\subset \Omega_1\subset T$ (see \cite{Win}). 

Let $E'=\mathbb C /\Gamma'$ and $E''=\mathbb C /\Gamma''$  be elliptic curves and $T=E'\times E''=\mathbb C^2/\Gamma$, where $\Gamma'=\mathbb Z\oplus (2\pi i \mathbb Z)$, $\Gamma''=(\sqrt{2}\mathbb Z)\oplus (2\pi i \mathbb Z)$, 
$\Gamma=\Gamma'\times \Gamma''.$  Let $\pi': \mathbb C\to E'$, $\pi'': \mathbb C\to E''$ and $\pi=(\pi',\pi''): \mathbb C^2 \to T$ be the natural projections. Then we see that $E'$ is not isogenous to $E''$. The compact complex torus $T$ carries a hermitian metric $h$ induced by the euclidean metric on $\mathbb C^2$ (i.e. $h=dz_1 \otimes d\bar z_1 +dz_2 \otimes d\bar z_2$). The associated distance function is said to be 
$d$ and the injectivity radius $\rho$ is given by $\rho=\frac{1}{2}\min_{\gamma\in \Gamma\setminus \{0\}}\|\gamma\|.$

We choose numbers $0<\rho'<\rho''<\rho$ and define 
$$
W=B_{\rho'}(E',e).
$$
Furthermore we choose $0<\delta<\rho/3$. Let $s: \mathbb C\to \mathbb C$ be a holomorphic function such that
$$
s(\overline{B_{\rho'}(\mathbb C, 0)}) \subset B_{3\delta}(\mathbb C, 0)
$$
and  $\mathrm{diam}(s(B_{\rho'}(\mathbb C, 0)) )>2\delta$.  Now let $\sigma=\pi''\circ s$. Then there exist complex numbers $t,t'\in B_{\rho'}(\mathbb C, 0)$ such that
$$
d_{E''}(\sigma(t), \sigma(t'))>2\delta.
$$

 Define $\Omega_2=(E'\setminus \overline{W})\times E''$ and $\Omega_1=\Omega_2\cup \Sigma$ with
$$
\Sigma=\{(x,y): x\in \overline{ W}, y\in E'', d_{E''}(y,\sigma(x))<\delta\}. 
$$

J. Winkelmann \cite{Win} showed the following proposition which is a slight improvement of Arakelyan's theorem.
\begin{proposition}\label{Prop1}
Let $B$ be a closed subset in $\mathbb C$ for which $\mathbb P_1\setminus B$ is connected and locally connected at $\infty$. Let $q$ be a point in the interior of $B$ and let $f:B\to \mathbb C$ be a continuous function which is holomorphic in the interior of $B$. Furthermore let $\epsilon: B\to \mathbb R^+$ be a continuous function. Then there exists an entire function $F$ such that
$$
F(q)=f(q),\; F'(q)=f'(q)\; \text{and}\; |F(z)-f(z)|<\epsilon(z)
$$
for all $z\in B$.
\end{proposition}

In \cite{Win}, J. Winkelmann proved the following theorem.

\begin{theorem}[J. Winkelmann] \label{Th1}  With notations as above one has
\begin{itemize}
\item[(i)] For every point $p\in \Omega_1$ and every $v\in T_p(\Omega_1)$ there is a non-constant holomorphic map $f: \mathbb C \to \Omega_1$ with $p=f(0), v=f'(0)$ and $\overline{\Omega}_1=\overline{f(\mathbb C)}$. 
\item[(ii)]If $f: \mathbb C\to T$ is a non-constant holomprphic map with bounded derivative (with respect to the euclidean metric on $\mathbb C$ and $h$ on $T$ ) and $f(\mathbb C)\subset \overline{\Omega}_1$, then $f(\mathbb C)\subset \overline{\Omega}_2$. Moreover, $f$ is affine-linear and $\overline{f(\mathbb C)}$ is closed analytic subset of $T$. 
\end{itemize}
\end{theorem}

Let $A$ be the union of $\overline{B_{\rho'}(\mathbb C, \gamma)}$ for all $\gamma\in \Gamma'$. 
Let $g: \mathbb C\setminus \{(-\infty, -1]\cup [1,+\infty)\} \to \{z\in \mathbb C:- \frac{\pi}{2}<\text{Im} z<\frac{\pi}{2}\}$ be the inverse of the biholomorphic function $\{z\in \mathbb C: -\frac{\pi}{2}<\text{Im} z<\frac{\pi}{2}\}\ni z\mapsto \sin (iz)\in \mathbb C\setminus \{(-\infty, -1]\cup [1,+\infty)\} $.

 For each $j\in \mathbb Z$ denote by $B_j=\sin ( iB_{2\rho'}(\mathbb C, j))$. We let $B=\cup_{j\in \mathbb Z} \bar B_j$. Denote by $h: B\to\mathbb C$ the continuous function given by $h(w)=s(g(w)-j)$ for any $j\in\mathbb Z$ and for any $w\in \bar B_j$. Note that $h$ is holomorphic in the interior of $B$. So, using Proposition \ref{Prop1} we deduce that there exists an entire function $F:\mathbb C\to \mathbb C$ such that
$$
|F(z)-h(z)|<\delta/3\;\text{for all}\; z\in B.
$$

\begin{lemma}  $\pi (z,F(\sin(iz))\in \Omega_1$ for all $z\in \mathbb C$.
\end{lemma}
\begin{proof}
It suffices to show that if $z\in A$, then $d_{E''}(\pi''(F\sin(iz)), \sigma(\pi'(z)))<\delta/3$. Indeed,
suppose that $z\in A$. Then there exist  integers $k,j\in \mathbb Z$ such that $|z-2k\pi i-j|\leq \rho'$. Therefore 
\begin{equation*}
\begin{split}
h(\sin(iz))&= h(\sin(i(z-2k\pi i)))=s(g(\sin(i(z-2k\pi i)))-j)\\
&=s(z-2k\pi i -j)=\sigma(\pi'(z)).
\end{split}
\end{equation*}
This implies that

\begin{equation*}
\begin{split}
d_{E''}(\pi''(F(\sin(iz))), \sigma(\pi'(z)))=|F(\sin(iz)-h(\sin(iz))|<\delta/3<\delta.
\end{split}
\end{equation*}
The lemma is proved.
\end{proof}

Now we let $f: \mathbb{C} \to (\mathbb C^*)^2$ be the holomorphic map given by 
$$f(z)=(\exp(z),\exp( F(\sin( iz))).$$
For each $k\in \mathbb N^*$ denote by $g_k: \Delta \to (\mathbb C^*)^2$ the holomorphic map defined by
$g_k(z):=f(kz)$ for all $z\in \Delta$. Since $g_k(0)=f(0)\in (\mathbb C^*)^2$ and ${g'}_k(0)=k f'(0)=k v$, where $v=f'(0)=(1, i F'(0)\exp(F(0))\ne 0$, ${g_k}$ is not normal and is not compactly divergent. 

Suppose that there exist a sequence $\{k_n\}\subset \mathbb N$, a sequence $\{p_n\}\Subset \Delta$, and a sequence $\{\rho_n\}\subset (0,+\infty) $ with $\rho_n \to 0^+ $ as $n\to \infty$ such that the sequence $\{\varphi_n\}$ defined by
$$
\varphi_n(\xi):=g_{k_n}(p_n+\rho_n \xi)=f(k_n p_n+k_n \rho_n \xi),
$$
for each $n\in \mathbb N^*$ and for $|\xi|< 1/\rho_n $, converges uniformly on every compact subset of $\mathbb C$ to a non-constant $ds^2_{FS}$-Brody curve $ \varphi: \mathbb C\to (\mathbb C^*)^2$, where $ds^2_{FS}$ is the Fubini-Study metric on $\mathbb P^2(\mathbb C)$.

Let $u,v: \mathbb C\to \mathbb C$ be holomorphic curves such that
 $$\varphi(z)=(\exp(u(z)), \exp(v(z)))$$ for all $z\in \mathbb C$. Since $\|\varphi'\|_{FS}$ is bounded, $T_r(\varphi)=O(r)$ (see \cite{C-H}), and thus the holomorphic functions $u$ and $v$ are both affine-linear.  
\begin{lemma}\label{lem1}
$\pi \circ (u(\mathbb C),v(\mathbb C)) \subset \Omega_1$.
\end{lemma}
\begin{proof}
Fix an arbitrary $z\in \mathbb C$.  We will show that $\pi \circ (u(z),v(z))\in \Omega_1$. It suffices to prove that this assertion holds for the case $u(z)\in A$. Indeed, if $u(z)\in A$, then there exist integers $k^*,j^*\in \mathbb Z$ such that $|u(z)-2k^*\pi i -j^*|\leq \rho'$. 

Since $\exp(k_np_n+k_n\rho_n z)\to \exp(u(z)) $ as $n\to \infty$, there is a sequence of integers $\{l_n\}\subset \mathbb Z$ such that $k_np_n+k_n\rho_n z-2 l_n \pi i \to u(z)$  as $n\to \infty$. Therefore 
\begin{equation*} 
\begin{split}
\exp(v(z))&=\lim_{n\to \infty} \exp(F(\sin(i(k_np_n+k_n\rho_n z))))\\
&=\lim_{n\to \infty} \exp(F(\sin(i(k_np_n+k_n\rho_n z-2l_n \pi i))))\\
&= \exp(F(\sin(iu(z))))).
\end{split}
\end{equation*}
Hence there is an integer $l\in \mathbb Z$ such that $F(\sin(i u(z)))=v(z)+2l\pi i$.

Now we see that 
\begin{equation*} 
\begin{split}
h(\sin(iu(z)))&=h(\sin(i(u(z)-2k^*\pi i)))= s(g(\sin(i(u(z)-2k^*\pi i)))-j^*)\\
                   & = s(u(z)-2k^*\pi i-j^*)=\sigma(\pi'(u(z))).
\end{split}
\end{equation*}
Thus we conclude that
\begin{equation*}
\begin{split}
d_{E''}(\pi''(v(z)), \sigma(\pi'(u(z))))&=d_{E''}(\pi''(F(\sin(i u(z)))), \sigma(\pi'(u(z))))\\
&=|F(\sin(iz)-h(\sin(iz))|<\delta,
\end{split}
\end{equation*}
and hence $(\pi'(u(z)), \pi''(v(z)))\in \Sigma$. The proof is complete.

\end{proof}

\begin{lemma}\label{lem2} $\pi \circ (u(\mathbb C),v(\mathbb C))\cap (\Omega_1\setminus\overline{\Omega}_2)\not= \emptyset$.
\end{lemma}
\begin{proof}
Now we consider two following cases.
\smallskip

\underbar{Case (\romannumeral1)}: $k_n\rho_n \not \to 0$ as $n\to \infty$. Without loss of generality, we may assume that $k_n\rho_n\gtrsim 1$ for all $n\in \mathbb N$. Therefore there are a positive integer $n_0$ and a positive real number $R>0$ such that for each $n\geq n_0$ there exists a point $\xi_n \in \mathbb C$ with $|\xi_n|\leq R$ such that $k_n p_n+k_n \rho_n \xi_n \in A'\subset A$, where $A'$ is the union of $\overline{B_{\rho'/2}(\mathbb C, \gamma)}$ for all $\gamma\in \Gamma'$. Without loss of generality, we may assume that $\xi_n \to \xi^*\in \overline{B_R(\mathbb C,0)}$ as $n\to \infty$. Since $\exp(k_n p_n+k_n \rho_n \xi)$ converges uniformly on $\overline{B_R(\mathbb C,0)}$ to $\exp(u(\xi))$, there exists a sequence $\{l_n\}\subset \mathbb N$ such that  $k_n p_n+k_n \rho_n \xi_n+ 2l_n\pi i \to u(\xi^*)\in A'$ as $n\to \infty$. Thus by Lemma \ref{lem1}, $\pi(u(\xi^*),v(\xi^*))\in \Omega_1\setminus\overline{\Omega}_2 $ and hence $\pi \circ (u(\mathbb C),v(\mathbb C))\cap (\Omega_1\setminus\overline{\Omega}_2)\not= \emptyset$.
 
\smallskip

\underbar{Case (\romannumeral2)}: $k_n\rho_n \to 0$ as $n\to \infty$.  
Consider two following subcases.
\smallskip

\underbar{Subcase (\romannumeral2.1)}: $\{k_np_n\} \Subset \mathbb C$. With no loss of generality, we may assume that there exist a complex number $\alpha \in \mathbb C$ and a positive real number $R>0$ such that  $|k_np_n- \alpha| <R$ for all $n\in \mathbb N^*$.  Then we get
$$
\sup_{|\xi|\leq1} \|{\varphi_n}'(\xi)\|=k_n\rho_n\sup_{|\xi|\leq1} \|{f}'(k_np_n+k_n\rho_n\xi)\| \leq k_n\rho_n\sup_{|t-\alpha|\leq R+1} \|f'(t)\| \to 0
$$
as $n\to \infty$. This implies that $\varphi' \equiv 0$, and thus $\varphi$ is constant. It is impossible.

\smallskip

\underbar{Subcase (\romannumeral2.2)}: $k_np_n \to \infty$ as $n\to \infty$. 
If $\text{Re}(k_np_n)\to -\infty$, then 
\begin{equation*}
\begin{split}
|\exp(k_np_n+k_n\rho_n\xi)|&=\exp(\text{Re}(k_np_n)+k_n\rho_n\xi)\\
&=\exp(\text{Re}(k_np_n)) \exp(k_n\rho_n\xi) \to 0 
\end{split}
\end{equation*}
 as $n\to \infty$ for each $\xi \in \mathbb C$. Thus  $\varphi^1(\xi)=0$ for all $\xi \in \mathbb C$, where $\varphi^1$ is the first component of the holomorphic map $\varphi$. It is not possible.

If  $\text{Re}(k_np_n)\to +\infty$, then $|\exp(k_np_n+k_n\rho_n\xi)|=\exp(\text{Re}(k_np_n)+k_n\rho_n\xi)=\exp(\text{Re}(k_np_n)) \exp(k_n\rho_n\xi) \to +\infty $ as $n\to \infty$ for each $\xi \in \mathbb C$. This is a contradiction.

If  $|\text{Re}(k_np_n)|\lesssim 1 $ for all $n\in \mathbb N$. In this case we may assume that   there exists a positive number $R>0$ and integers $l_n\in \mathbb Z$, $n=1,2,\ldots$, such that
$$ 
|k_np_n+2l_n\pi i|\leq R 
$$
for every $n\in \mathbb N$. We note that $f(k_np_n+k_n\rho_n\xi)=f(k_np_n+2l_n \pi+ k_n\rho_n\xi)$ for all $n\in \mathbb N$ and for $\xi \in \mathbb C$.

Therefore we get
$$
\sup_{|\xi|=1} |{\varphi_n}'(\xi)|= k_n\rho_n\sup_{|\xi|=1}|f'(k_np_n+2 l_n \pi  i+k_n\rho_n\xi)|\leq k_n\rho_n \sup_{|t|<R+1} |f'(t)|\to 0 
$$
as $n\to \infty$. Hence $\varphi$ is constant, which is impossible.

\end{proof}

Now it follows that Lemma \ref{lem2} together with Theorem \ref{Th1} and Lemma \ref{lem1}  completes the proof of Theorem \ref{T2}.

\begin{remark}
By arguments as above there do not exist a sequence $\{k_n\}\subset \mathbb N$, a sequence $\{p_n\}\Subset \Delta$, and a sequence $\{\rho_n\}\subset (0,+\infty) $ with $\rho_n \to 0^+ $ as $n\to \infty$ such that the sequence $\{\varphi_n\}$ defined by
$$
\varphi_n(\xi):=g_{k_n}(p_n+\rho_n \xi)=f(k_n p_n+k_n \rho_n \xi),
$$
for each $n\in \mathbb N^*$ and for $|\xi|< 1/\rho_n $, converges uniformly on every compact subset of $\mathbb C$ to a non-constant $ds^2_{FS}$-Brody curve $ \varphi: \mathbb C\to (\mathbb C^*)^2$ given by $\varphi(z)=(\exp(u(z)), \exp(v(z)))$ for all $z\in \mathbb C$. However, by \cite[Th\'{e}or\`{e}me 1.12, p. 440]{Ber} there exist sequences $\{A_n\}, \{B_n\}\subset \mathbb C$ such that $f(A_n z+B_n)$ converges uniformly on every compact subset of $\mathbb C$ to a non-constant curve $\varphi$ in $(\mathbb C^*)^2$ given by $\varphi(z)=(\exp(az+b), \exp(cz+d)))$ for all $z\in \mathbb C$, where $a,b,c,d\in \mathbb C$ with $|a|^2+|c|^2\ne 0$.
\end{remark}

\section{Normal families of holomorphic mappings in several complex variables}

First of all, we recall some definitions.
\begin{definition}
A family $\mathcal F$ of holomorphic maps from a complex space $X$ to
a complex space $Y$ is said to be normal if $\mathcal F$ is relatively
compact in $\mathrm{Hol}(X,Y)$ in the compact-open topology.
\end{definition}
\begin{definition}
Let $X$, $Y$ be complex spaces and $\mathcal F \subset \mathrm{Hol}(X,Y)$.

i) A sequence $\big\{f_j\big\} \subset \mathcal F$ is compactly divergent
if for every compact set $K \subset X$ and for every compact set 
$L \subset Y,$ there is a number $j_0 = j(K,L)$ such that 
$f_j(K) \cap L = \emptyset$ for all $j \geq j_0$.

ii) The family $\mathcal F$ is said to be not compactly divergent if
$\mathcal F$ contains  no compactly divergent subsequences.
\end{definition}

To prove Theorem \ref{T1}, we need the following lemma (cf. see \cite[lemme 2.2, p. 431]{Ber}).

\begin{lemma}\label{L6} Let $(X,d)$ be a complete metric space and let $\varphi: X\to \mathbb R^+$ be a locally bounded function. Let $\epsilon>0$ and let $\tau>1$. Then, for all $a\in X$ satisfying  $\varphi(a)>0,$ there exists $\tilde a \in X$ such that

\begin{enumerate}
\item[(i)] $d(a,\tilde a )\leq \frac{\tau}{\epsilon\varphi(a) (\tau-1)}$
\item[(ii)] $\varphi(\tilde a)\geq \varphi(a)$
\item[(iii)] $\varphi(x)\leq \tau \varphi(\tilde a)$ if $d(x,\tilde a)\leq \frac{1}{\epsilon \varphi(\tilde a)}$.
\end{enumerate}
\end{lemma}

\begin{proof}[Proof of Theorem \ref{T1}]\

\noindent
{\bf ($\Rightarrow$)}\ Consider two cases

\noindent
{\bf Case 1.} The family $\mathcal{F}$ is compactly divergent.

Then there is a sequence $\{f_j\}\subset \mathcal{F}$ such that  $\{f_j\}$ is compactly divergent. Take $p_0\in \Omega$ and $r_0>0$ such that $B(p_0,r_0)\Subset \Omega$. Take $p_j=p_0$ for all $j\geq 1$ and $\rho_j>0$ for all $j\geq 1$ such that $\rho_j\to 0^+$ as $j\to \infty$ and 
$$
g_j(\xi)=f_j(p_j+\rho_j\xi)\; \text{for all}\; j\geq 1.
$$
Note that $g_j$ is defined on 
$$
\{\xi \in \mathbb C: |\xi|\leq R_j:=\frac{1}{\rho_j} dist(p_0, \partial \Omega)\}.
$$
Assume that $K$ is a compact subset of $\mathbb C$ and $L$ is a compact subset of $M$. Then there exists $j_0\geq 1$ such that $p_j+\rho_j K\subset B(p_0,r_0)$ for all $j\geq j_0$. This implies that $g_j(K)\subset f_j(\bar B(p_0,r_0))$ for each $j\geq j_0$.

Since the sequence $\{f_j\}$ is compactly divergent, there exists $j_1>j_0$ such that 
$ f_j(\bar B(p_0,r_0))\cap L=\emptyset$ for all $j\geq j_1$. Thus $g_j(K)\cap L=\emptyset$ for all $j\geq j_1$. This means that $\{g_j\}$ is compactly divergent.

\noindent
{\bf Case 2.} The family $\mathcal{F}$ is not compactly divergent.

By Lemma 2.6 in \cite{Do}, there exist sequences $\{f_k\}\subset \mathcal{F}$, $\{a_k\}\Subset \Omega$ such that 
$$
|{f_k}'(a_k)|_E\geq k^3  
$$
for all $k\geq 1$. For simplicity, we denote by $|.|:=|.|_E$. Without loss of generality, we may assume that $a_k\to a_0\in \Omega$ as $k\to \infty$ and $\overline{B(z_0,r)}\subset \Omega$ for some $r>0$. We also can assume that $a_k\in \overline{B(z_0,r)}$ for all $k\geq 1$. We denote by $d_E$ the distance induced by the hermitian metric $E$. 

Now by applying Lemma \ref{L6} to $X=\overline{B(z_0,r)}$, $\varphi:=|{f_k}'|_E, a=a_k$ and $\tau=1+\frac{1}{k},$ it implies that there exists $\tilde a=: z_k$ such that 
\begin{itemize}
\item[(i)] $|z_k-a_k|\leq \frac{\tau}{\epsilon \varphi(a) (\tau-1)}\leq \frac{2 k^2}{|{f_k}'(a_k)|}\leq \frac{2}{k}$;
\item[(ii)] $|{f_k}'(z_k)|=\varphi(\tilde a)\geq \varphi(a)=|{f_k}'(a_k)| \geq k^3$;
\item[(iii)] $|{f_k}'(z)|=\varphi(z)\leq \tau \varphi(\tilde a)=(1+\frac{1}{k})|{f_k}'(z_k)|$ for all $|z-z_k|\leq \frac{1}{\epsilon \varphi(\tilde a)}=\frac{k}{|{f_k}'(z_k)|}$.
\end{itemize}
Let $\rho_k:=\frac{1}{|{f_k}'(z_k)|}$ and let $g_k(z):=f_k(z_k+\rho_k z)$. By $(i),$ we have $z_k\to z_0$ as $k\to\infty$. Therefore $g_k$ is defined on $\Delta_k:=\{z\in \mathbb C: |z|<k\}$ for all $k$ big enough. Moreover, by $(ii),$ we get $k\rho_k\leq \frac{1}{k^2}$. Now because of $(iii),$ we obtain 
$$
|{g_k}'(z)|=\rho_k |{f_k}'(z_k+\rho_k z)| \leq (1+\frac{1}{k})\; \text{for all}\; z\in \Delta_k.
$$ 
So the $g_k$ are holomorphic on larger and larger discs in $\mathbb C$ and they have bounded derivatives. Thus the family $\{g_k\}$ is equicontinuous. If the family $\{g_k\}$ is not compactly divergent, by a result of Wu \cite[Lemma 1.1.iii]{Wu}, it is normal. This implies that there exists a subsequence $\{g_{k_j}\}\subset \{g_k\}$ such that $\{g_{k_j}\}$ converges uniformly on any compact subset of $\mathbb C$ to a holomorphic map $g:\mathbb C\to M$. It is easy to see that 
$$
|g'(0)|=\lim_{j\to \infty} |{g_{k_j}}'(0)|=\lim_{j\to \infty}\rho_{k_j} |{f_{k_j}}'(z_{k_j})|=1.
$$
Moreover $|g'(z)|=\lim_{j\to \infty} |{g_{k_j}}'(z)|=\lim_{j\to \infty}\rho_{k_j} |{f_{k_j}}'(z_{k_j}+\rho_{k_j} z)|\leq \lim_{j\to\infty}(1+\frac{1}{k_j})=1$ for all $z\in \mathbb C$. This implies that $g$ is a non-constant $E$-Brody curve in $M$. 

\noindent
{\bf ($\Leftarrow $)}\  Suppose that the family $\mathcal{F}$ is normal. 

Now we consider two cases.

\noindent
{\bf Case 1.} The sequence $\{g_j\}$ converges uniformly on any compact subset of $\mathbb C$ to a non-constant holomorphic map $g:\mathbb C\to M$. 

Take $r_0>0$ such that $B(p_0,r_0)\Subset \Omega$. Without loss of generality, we may assume that $\{p_j\}\subset B(p_0,r_0) $. Put $K_0=\overline{B(p_0,r_0)}\subset \Omega$. Since $\mathcal{F}$ is normal, by \cite[Lemma 2.6 (i), p.472]{Do}, there exists a constant $N>0$ such that 
$$
\sup_{p\in K_0}|f'(p)|\leq N \; \text{for each}\; f\in \mathcal{F}.
$$
Fix $\xi\in \mathbb C$. Then $p_j+\rho_j \xi\in K_0$ for $j$ large enough. Hence
$$
|{g_j}'(\xi)|=\rho_j |{f_j}'(p_j+\rho_j \xi)| \leq \rho_j N.
$$

Taking the limit, we obtain 
$$
|{g}'(\xi)|=\lim_{j\to \infty} |{g_j}'(\xi)|=0.
$$
This implies that $g$ is constant. This is impossible. 

\noindent
{\bf Case 2.} The sequence $\{g_j\}$ is compactly divergent. 

Since the family $\mathcal{F}$ is normal, without loss of generality, we may assume that the sequence $\{f_j\}$ converges uniformly on any compact subset of $\Omega$ to $f\in \mathrm{Hol}(\Omega, M)$. For $\xi\in \mathbb C,$ we have
$$
g_j(\xi)=f_j(p_j+\rho_j \xi) \to f(p_0)\in M
$$
since $\rho_j\to 0^+$ as $j\to\infty$. This implies that $\{g_j\}$ is not compactly divergent. This is a contradiction.

\end{proof}

In order to prove Theorem \ref{T3}, we need the following lemmas.

\begin{lemma}\label{lem4} Let $Z$ be a complex manifold. Let $S$ be a complex hypersurface of a complex space $X$. If a sequence $\{\varphi_n\}\subset \mathrm{Hol}(Z,X\setminus S))$ converges uniformly on every compact subset of $Z$ to a mapping $\varphi \in \mathrm{Hol}(Z,X)$, then either $\varphi(Z)\subset X\setminus S$ or $\varphi(Z)\subset S$.
\end{lemma}
\begin{proof}
Suppose  $\varphi(Z)\cap S\ne 0$. Put $\tilde Z:=\{z\in Z: \varphi(z)\in S\}$. Then $\tilde Z\ne \emptyset$. 

 Since $S$  is a closed set in $X$, it is easy to see that $\tilde Z$ is closed in $Z$. Moreover, $\tilde Z$ is open. Indeed, let $z_0\in \tilde Z$. Then there exist an open neighborhood of $\varphi(z_0)$ and a holomorphic function $f\in \mathrm{Hol}(U,\mathbb C)$ such that 
$$ U\cap S=\{w\in U: f(w)=0\} .$$
Since $\{\varphi_n\}$ uniformly converges to $\varphi$ on every compact subset of $Z$, there exists an open neighborhood $W$ of $z_0$ such that $\varphi_n(W)\subset U\setminus S $ for $n\geq l$. Moreover, $\{f\circ \varphi_n\}\subset \mathrm{Hol}(W,\mathbb C)$ converges uniformly to the holomorphic function $f\circ \varphi\in \mathrm{Hol}(W,\mathbb C)$. Since $f\circ\varphi(z_0)=0$, by Hurwitz's theorem, $f\circ \varphi \equiv 0$ on $W$. Therefore, $z_0\in W \subset \tilde Z$ and thus $\tilde Z$ is open.

By the connectivity of $Z$, we obtain that $\tilde Z=Z$. This completes the proof. 
\end{proof}

\begin{lemma}
Let $\Omega$ be a domain in $\mathbb C^m$. Let $S$ be a complex hypersurface in a  compact complex manifold $X$ with a hermitian 
metric $E$  and let $M=X\setminus S$. Let $\mathcal{F}\subset \mathrm{Hol}(\Omega, M)$ such that  $\mathcal{F}$ is not compactly divergent. 
Then, the family $\mathcal{F}$ is normal if and only if for each compact subset $K$ of $\Omega$, there is a constant $c_K>0$ such that
\begin{equation}\label{eq1105}
E(f(z),df(z)(\xi))\leq c_K |\xi| \;\text{ for every }\; z\in K, \; \xi \in \mathbb C^m \setminus\{0\},\; f\in \mathcal{F}.
\end{equation}
\end{lemma}
\begin{proof}

\noindent
{\bf ($\Rightarrow$).} We will show that (\ref{eq1105}) holds. Suppose that, on the contrary, there exist a compact subset $K\subset \Omega$ and a sequences $\{z_k\}\subset K$, $\{\xi_k\}\subset \mathbb C^m\setminus \{0\}$ with $|\xi_k|=1$ for all $k\geq 1$, $\{f_k\}\subset \mathcal{F}$ such that 
$$
E(f(z_k),df_k(z_k)(\xi_k))\to \infty\;\text{as}\; k\to \infty.
$$ 
Without loss of generality we may assume that $z_k\to z_0\in K$ and $\xi_k\to \xi_0$ as $k\to \infty$. Moreover, since $\mathcal{F}$ is normal we can assume that $\{f_k\}$ converges uniformly on $K$ to $f_0\in \mathrm{Hol}(\Omega, M)$ as $k\to \infty$. Therefore we get
$$
E(f_0(z_0),df_0(z_0)(\xi_0))=\lim_{k\to\infty} E(f(z_k),df_k(z_k)(\xi_k))=\infty.
$$
This is a contradiction.

\noindent
{\bf ($\Leftarrow$).}
Suppose that $(\ref{eq1105})$ holds. We will show that $\mathcal{F}$ is normal. Indeed, given $x_0\in \Omega$. Take $r>0$ be such that
$$
K:=\overline{B}(x_0,r)\subset \Omega.
$$
Since $K$ is compact, by the hypothesis, this is a constant $c_K>0$ such that 
$$
E(f(z),df(z)(\xi))\leq c_K |\xi| \;\text{ for every }\; z\in K, \; \xi \in \mathbb C^m \setminus\{0\},\; f\in \mathcal{F}.
$$  
For every $x\in K$ consider the curve $\gamma: [0,1] \to \Omega$ joining $x_0$ and $x$ given by $\gamma(t)=t(x-x_0)+x_0$. Then
\begin{equation*}
\begin{split}
 dist(f(x),f(x_0))&\leq \int_{0}^1 E(f\circ \gamma(t),d(f(\circ \gamma)(t))dt\\
&\leq c_K \int_{0}^1 | \gamma'(t)|dt=C_K .\|x-x_0\|\; \text{for all }\; f\in \mathcal{F}.
\end{split}
\end{equation*}
This implies that the family $\mathcal{F}$ is equicontinuous. Therefore, by Ascoli's theorem $\mathcal{F}$ is normal in $\mathrm{Hol}(\Omega, X)$. Now it suffices to show that if a sequence $\{f_k\}\subset \mathcal{F}$ converges to a holomorphic map $f\in \mathrm{Hol}(\Omega,X)$. Indeed, it follows from Lemma \ref{lem4} that either $f(\Omega)\subset S$ or $f(\Omega)\subset X\setminus S=M$. Since $\mathcal{F}$ is not compactly divergent, the case $f(\Omega)\subset S$ does not occur. Thus the proof is complete.
\end{proof}

\noindent
{\it Proof of Theorem \ref{T3}}

Since $X$ is compact, by Zalcman's lemma (cf. see \cite[Lemme 2.1, p.430]{Ber}) we have that the family $\mathcal{F}$ is not normal if and only if there exist sequences $\{p_j\}\subset \Omega$ with $p_j\to p_0\in\Omega$ as $j\to\infty$, $\{f_j\} \subset \mathcal{F}, \{\rho_j\}\subset \mathbb R$ with $\rho_j>0$ and $\rho_j \to 0^+$ as $j\to \infty$ such that 
$$
g_j(\xi):=f_j(p_j+\rho_j\xi), \xi\in \mathbb  C
$$
converges uniformly on every compact subset of $\mathbb C$ to a non-constant $E$-Brody curve $g: \mathbb C\to X$. 

Because $S$ is a hypersurface in $X$, by Lemma \ref{lem4} we conclude that either $g(\mathbb C)\subset X\setminus S=M$ or $g(\mathbb C)\subset S$. Thus one of the following two assertions holds
\begin{itemize}
\item[(i)] The sequence $\{g_j\}$ is compactly divergent on $\mathbb C$;
\item[(ii)] The sequence $\{g_j\}$ converges uniformly on compact subsets of $\mathbb C$  to a non-constant $E$-Brody curve $g: \mathbb C\to M$. 
 \end{itemize}
\hfill $\boxed{ }$

\begin{corollary}\label{Co11}
Let $\{f_n: \Delta\to \mathbb C^*\}$ be a non-normal sequence of holomorphic functions. If $\{f_n\}$ is not compactly divergent, then there exist a subsequence $\{f_{n_j}\}\subset\{f_n\}$ and sequences $\{p_j\}\subset \Delta$ with $p_j\to p_0\in \Delta$ as $j\to \infty$ and $\{\rho_j\}\subset \mathbb R^+$ with $\rho_j\to 0^+$ as $j\to \infty$ such that the following sequence
$$
g_j(\xi):=f_{n_j}(p_j+\rho_j\xi)
$$
converges uniformly on every compact subset of $\mathbb C$ to $E(\xi)=\exp(A_0 \xi+B_0)$, where $A_0\in \mathbb C^* $ and $B_0\in \mathbb C$.
\end{corollary}
\begin{proof}
Let $\{f_n: \Delta\to \mathbb C^*\}$ be a  non-normal sequence of holomorphic functions which is not compactly divergent. Since $\mathbb C^*=\mathbb P^1(\mathbb C)\setminus \{0,\infty\}$, by Theorem \ref{T3} there exist a subsequence $\{f_{n_j}\}\subset\{f_n\}$ and sequences $\{p_j\}\subset \Delta$ with $p_j\to p_0\in \Delta$ as $j\to \infty$ and $\{\rho_j\}\subset \mathbb R^+$ with $\rho_j\to 0^+$ as $j\to \infty$ such that the following sequence
$$
g_j(\xi):=f_{n_j}(p_j+\rho_j\xi)
$$
converges uniformly on every compact subset of $\mathbb C$ to $\rho$-Brody curve $g:\mathbb C \to \mathbb C^*$, where $\rho$ is the spherical metric on $\mathbb P^1(\mathbb C)$. Moreover because the spherical derivative of $g$ is bounded, $g(\xi)=E(\xi)=\exp(A_0 \xi+B_0)$ for all $\xi \in\mathbb C$, where $A_0\in \mathbb C^* $ and $B_0\in \mathbb C$.
\end{proof}

F. Berteloot and J. Duval proved the following theorem (see \cite[Lemme 2.4, p. 434]{Ber}).
\begin{theorem}[F. Berteloot - J. Duval]\label{Th5} Let $f: \mathbb C\to \mathbb P^1(\mathbb C)\setminus \{0,\infty\}$ be a non-constant holomorphic function. Then there exist sequences $\{A_k\}\subset \mathbb C$ and $\{B_k\}\subset \mathbb C$ such that $f(A_k z+B_k)$ converges uniformly on every compact subsets of $\mathbb C$ to $E(z)=\exp(A_0 z+B_0)$, where $A_0\in \mathbb C^* $ and $B_0\in \mathbb C$.
\end{theorem}

The proof of the above theorem is given in \cite{Ber}. The following is a different proof which is an application of Corollary \ref{Co11}.

\begin{proof}[Proof of Theorem \ref{Th5}]
Suppose that $f: \mathbb C\to \mathbb P^1(\mathbb C)\setminus \{0,\infty\}=\mathbb C^*$ is a non-constant holomorphic function. Then there is a point $a_0\in\mathbb C$ such that $f'(a_0)\ne 0$. Without loss of generality $ a_0=0$. For each $k\in \mathbb N^*$, define $f_k:\Delta\to \mathbb C^*$ by $f_k(z)=f(kz)$ for all $z\in \Delta$. Since $f_k(0)=f(0)\in \mathbb C^*$ and ${f_k}'(0)=k f'(0)\to \infty$ as $k\to \infty$, $\{f_k\}$ is not normal and is not compactly divergent. Thus, by Corollary \ref{Co11} there exist sequences $\{k_n\}\subset\mathbb N$, $\{p_n\}\subset \Delta$ with $p_n\to p_0\in \Delta$ as $n\to \infty$, and $\{\rho_n\}\subset \mathbb R^+$ with $\rho_n\to 0^+$ as $n\to \infty$ such that the following sequence
$$
g_n(\xi):=f_{k_n}(p_n+\rho_n\xi)=f(k_np_n+k_n\rho_n \xi)
$$
converges uniformly on every compact subset of $\mathbb C$ to $E(\xi)=\exp(A_0 \xi+B_0)$, where $A_0\in \mathbb C^* $ and 
$B_0\in \mathbb C$. So, the proof is complete.
\end{proof}

\begin{example}
Let $\{k_j\}\subset \mathbb Z$ be such that $e^{ik_j}\to 1$ as $j\to \infty$. Then the sequence $\{g_j\}$ given by
$$
g_j(z):=\exp(\exp(i\frac{\pi}{2}+\frac{z}{k_j}+\ln k_j))
$$
converges uniformly on every compact subset of $\mathbb C$ to the holomorphic function $g:\mathbb C\to \mathbb C^*$ given by $g(z)=e^{iz}$ for all $z\in \mathbb C$. Indeed,
we have 
\begin{equation*}
\begin{split}
g_j(z)&=\exp(\exp(i\frac{\pi}{2}+\frac{z}{k_j}+\ln k_j))\\
&= \exp(i k_j \exp(\frac{z}{k_j}))\\
&=\exp(i k_j (1+z/k_j+O(1/k_j^2)))\\
&=e^{ik_j} e^{iz +O(1/k_j)}.
\end{split}
\end{equation*}
We note that $\{O(1/k_j)\}$ converges uniformly on every compact subset of $\mathbb C$ to $0$. Thus $\{g_j\}$ converges uniformly on every compact subset of $\mathbb C$ to the holomprhic function $g:\mathbb C\to \mathbb C^*$ given by $g(z)=e^{iz}$ for all $z\in \mathbb C$.
\end{example}

\bibliographystyle{amsplain}

\begin{thebibliography}{99}
\vspace{20pt}

\bibitem{Al} L. Ahlfors, {\it Complex Analysis}, Third edition. International Series in Pure and Applied Mathematics. McGraw-Hill Book Co., New York, 1978.

\bibitem{A-K} G. Aladro and S. Krantz, {\it A criterion for normality in $\mathbb C^n$}, J. Math. Anal. Appl. 161 (1)(1991), 1--8.
 \bibitem{Ber} F. Berteloot, {\it M\'{e}thodes de changement d'\'{e}chelles en analyse complexe}, Ann. Fac. Sci. Toulouse Math. (6) 15 (2006), no. 3, 427--483.
\bibitem{B-E} M. Barrett and A. Eremenko, {\it Generalization of a theorem of Clunie and Hayman}, Proc. Amer. Math. Soc. 140 (4) (2012),  1397--1402.

\bibitem{R.B} R. Brody, {\it Compact manifolds and hyperbolicity}, Trans. Amer. Math. Soc., 235(1978), 213--219.
 \bibitem{C-D} B. Da Costa and J. Duval, {\it Sur les courbes de Brody dans $\mathbb P^n(\mathbb C)$}, Math. Ann. 355 (4) (2013), 1593--1600.
\bibitem{C-H} J. Clunie and W. K. Hayman, {\it The spherical derivative of integral and meromorphic functions}, Comment. Math. Helv. 40(1966), 117--148.
\bibitem{Do} D. T. Do, N .T. T. Pham and D. H. Pham, {\it Families of normal maps in Several Complex Variables and hyperbolicity of complex spaces}, Complex Variables, Vol. 48, No. 6 (2003), 469--482.
\bibitem{Duval} J. Duval, {\it Sur le lemme de Brody}, Invent. Math. 173 (2008), 305--314.
\bibitem{E1} A. Eremenko, {\it Brody curves omitting hyperplanes}, Ann. Acad. Sci. Fenn. Math. 35(2) (2010), 565--570.
\bibitem{Lang} S. Lang, {\it Introduction to Complex Hyperbolic Spaces}, Springer-Verlag, New York, 1987. 
\bibitem{L-V} O. Lehto and K. I. Virtanen, {\it Boundary behaviour and normal meromorphic fucntions}, Acta. Math. 97 (1957), 47--65.
\bibitem{Montel} P. Montel, {\it Sur les suites infinies des fonctions}, Ann. \'{E}cole Norm. Sup. 24 (1907), 233--334.
\bibitem{N-W} J. Noguchi and J. Winkelmann, {\it Nevanlinna theory in several complex variables and Diophantine
approximation}, Grundlehren der mathematischen Wissenschaften, Vol. 350, Springer, 2013.
\bibitem{Tsu} M. Tsukamoto, {\it A packing problem for holomorphic curves}, Nagoya Math. J. 194 (2009), 33--68
 \bibitem{TT}  Nguyen Van Trao and Pham Nguyen Thu Trang, {\it  On Zalcman complex spaces and Noguchi-type convergence-extension theorems for holomorphic mappings into weakly Zalcman complex spaces}, Acta Math. Vietnam. 32 (2007), 83--97. 
\bibitem{Win} J. Winkelmann, {\it On Brody and entire curves}, Bull. Soc. Math. France 135 (1) (2007), 25--46.
\bibitem{Win1} J. Winkelmann, {\it On Meromorphic Functions which are Brody Curves}, arXiv:0709.3929.
\bibitem{Wu} H. H. Wu, {\it Normal families of holomorphic mappings},  Acta Math. 119 (1967) 193--233.
\bibitem{Zal} L. Zalcman, {\it Normal families: new perspectives}, Bull. Amer. Math. Soc. (N.S) 35 (1998), 215--230.

\end{thebibliography}

\end{document}